\documentclass[a4paper,reqno]{amsart}

\usepackage{amssymb,amsmath,mathtools,mathrsfs}
\usepackage[hidelinks]{hyperref}
\usepackage{scalerel,stackengine}
\stackMath
\newcommand\reallywidehat[1]{%
\savestack{\tmpbox}{\stretchto{%
  \scaleto{%
    \scalerel*[\widthof{\ensuremath{#1}}]{\kern-.6pt\bigwedge\kern-.6pt}%
    {\rule[-\textheight/2]{1ex}{\textheight}}
  }{\textheight}%
}{0.5ex}}%
\stackon[1pt]{#1}{\tmpbox}%
}

\newtheorem{theorem}{Theorem}[section]
\newtheorem{lemma}[theorem]{Lemma}

\theoremstyle{definition}
\newtheorem{definition}[theorem]{Definition}

\newtheorem{prop}[theorem]{Proposition}

\theoremstyle{remark}
\newtheorem{remark}[theorem]{Remark}

\numberwithin{equation}{section}

\begin{document}

\title[time singularities for NSE]{On the possible time singularities for the 3D Navier-Stokes equations}


\author{Xiaoyutao Luo}
\address{Department of Mathematics, Statistics and Computer Science,
University of Illinois At Chicago, Chicago, Illinois 60607}
\email{xluo24@uic.edu}


\date{\today}
\begin{abstract}
We prove a local-in-time regularity criterion for the 3D Navier-Stokes equations. In particular, it follows from the criterion that the Hausdorff dimension of possible singular times of Leray-Hopf weak solutions $u\in L^r_t B^\alpha_{s,\infty}$ for some $\alpha>0$, $ s > 3$ and $r> 2 $ is less than $\frac{r}{2}(\frac{3}{s} + \frac{2}{r} -\alpha-1 )$. The main contribution is that we do not assume the suitability of weak solutions.
\end{abstract}

\maketitle
\section{Introduction}
Consider the 3D Navier-Stokes equations in the whole space
\begin{equation}\label{3D_NSE}
\begin{aligned}
\frac{du}{dt}  + (u\cdot \nabla)u - \Delta u & = -\nabla p\\
\nabla \cdot u  =0  ,
\end{aligned}
\end{equation}
where $u$ is the unknown vector field that describes velocity of the flow and $p$ is the scalar function that stands for the pressure of the fluid. The problem is supplemented by divergence free data $u_0 \in L^2(\mathbb{R})^3$.

Even though weak solutions have been constructed via various methods, the global regularity of \eqref{3D_NSE} remains open. Extensive studies of global regularity had been initiated but only conditional or partial results are available. For example, if $u\in L^s_tL^p_x$ for some $\frac{2}{r}+\frac{3}{s} \leq 1$ $s >3$ then the solution is regular \cite{LERAY}. There is a long history of improvements of this conditional regularity result. The limit case above $s=3$ is solved by Escauriaza, Seregin and \v{S}ver\'ak in \cite{ESS}. Note that their result is actually local and we will talk about this below. 

Since it is very difficult to prove the regularity of weak solutions, the theory of partial regularity of the solutions of \eqref{3D_NSE} arises, which focuses on estimating the size of the singular set in space and time. There is locality nature in this matter and instead of proving regularity in the whole $(0,T)\times \mathbb{R}^3$ local regularity results or criteria are considered. 

\begin{definition}
Given a weak solution $u$ of \eqref{3D_NSE}, the singular set $\mathcal{S}(u) \subset \mathbb{R}^+ \times \mathbb{R}^3$ is the set in which $u(x,t)$ is not locally bounded.
\end{definition}

Due to the parabolic nature of \eqref{3D_NSE}, it is natural to consider local regularity on parabolic cylinder $Q_r(x,t) = B_r(x) \times [t-r^2,t]$. At first glance $L^\infty$ does not seem to be very regular, but this definition makes sense due to the classical result of Serrin\cite{SERR} in which he proved that if $u\in L^r_tL^s_x(Q_r)$ for $\frac{2}{s}+\frac{3}{s} <1 $ then $\partial^k_x u(x,t) \in C^\alpha(Q_\frac{r}{2})$ for some $\alpha>0$ and any $k>0$. Later this result was improved by Struwe \cite{STRU} requiring only $\frac{2}{r}+\frac{3}{s} =1$ for $s <\infty$ and extended to $s=\infty$ in \cite{ESS}.

The following partial regularity results are known. The set of singular times, the projection of $\mathcal{S}(u)$ have zero $\frac{1}{2}$-Hausdorff measure $\mathcal{H}^\frac{1}{2}( \Pi_t\mathcal{S}(u) )=0$ for weak solutions satisfying the energy inequality including Leray-Hopf weak solutions, see for example \cite{SCHE}. The energy inequality 
$$
\|u(t) \|_2^2 +2\int_{t_0}^t \|\nabla u(s) \|_2^2 ds \leq \|u(0) \|_2^2
$$
for a.e $t_0$ and all $t>t_0$ is crucial here because it guarantees the uniqueness of strong solution in the class of Leray-Hopf weak solutions. This result dates back to Leray but it was implicit there. The latest attention in this field was brought to us by Scheffer, which leads to the well-known theorem of Caffarelli, Kohn and Nirenberg \cite{CKN}, the best partial regularity result so far. After introducing the notion of suitable weak solutions that satisfy the local energy inequality, they prove that the $1$-dimensional parabolic Hausdorff measure of $\mathcal{S}(u)$ is zero. The solutions constructed by Leray are suitable but the suitability Leray-Hopf weak solutions constructed by Galerkin approximation is unknown. Without the help of the local energy inequality it is extremely difficult to establish any local space-time regularity result. 

Before diving into the discussion of main results in this paper, let us briefly explain the issue of supercritically. The 3D Navier-Stokes equations are known to be supercritical, which means available \textit{a priori} bounds are not strong enough to control the higher norms of the solution. In order to guarantee regularity, one usually needs to impose some kind of condition that is subcritical or critical with respect to \eqref{3D_NSE}. Current techniques are not very effective in dealing with supercritical equations. The best possible result so far can only beat criticality by a logarithmic amount. See for instance \cite{BV,FJNZ,TAOGR}. 

Since there is little hope to overcome the supercriticality, we try to bridge the two ends of conditional regularity and partial regularity together. More precisely, we examine the following question: if we assume $u \in L^r_tL^s_x$ for some $\frac{3}{2}>\frac{2}{r}+\frac{3}{s} >1$ can we get a better bound on the dimension of time singularities? Notice by interpolation $\frac{2}{r}+\frac{3}{s}=\frac{3}{2}$ is satisfied for any weak solution.

In fact, as a direct consequence of a partial regularity result of Gustafson, Kang, and Tsai \cite{GKT}, for suitable weak solutions in the class $L^r_t L^s_x$, the set of possible singular times has the bound $\mathcal{H}^{\frac{r}{2}(\frac{3}{s} + \frac{2}{r} -1 )}(\Pi_t \mathcal{S}(u) ) =0$.\footnote{In fact, $\mathcal{P}^{r(\frac{3}{s} + \frac{2}{r} -1 )}(\mathcal{S}(u) ) =0$ where $\mathcal{P}^{s} $ is the $s$-dimensional parabolic measure.} Notice that when the parameters $\frac{3}{s} + \frac{1}{ r}< 1 $, the dimension $\frac{r}{2}(\frac{3}{s} + \frac{2}{r} -1 ) <  \frac{1}{2}$ improving the classical bound $ \frac{1}{2}$. However, the use of \textit{local energy inequality} is crucial for such types of  local regularity theory.

In this paper, we extend the result of \cite{GKT} on the Hausdorff dimension of the set of possible singular times. We consider Leray-Hopf weak solutions $u\in L^r_t B^\alpha_{s,\infty}$, $\alpha>0$, $ s > 3$ and $r> 2 $ without assuming the suitability of solutions. Our main results are as follows.

\begin{theorem}\label{main_thm}
Let $u $ be a Leray-Hopf weak solution of \eqref{3D_NSE}. Suppose $u\in L^r_t B^\alpha_{s,\infty}$ for some $\alpha>0$, $ s > 3$ and $r> 2 $, then for possible singular times we have 
$$
\mathcal{H}^{r(\frac{3}{s} + \frac{2}{r} -\alpha-1 )}(\mathcal{S}_T) =0.
$$
\end{theorem}

Theorem \ref{main_thm} is an application of the following local-in-time regularity criterion in terms of Besov norm. To the author's best knowledge, it is the first result of such type.
\begin{theorem}\label{main_thm_crite}
For any $ s > 3$ and $r> 2 $, there exist a constant $\delta >0$ with the following property: if a Leray-Hopf weak solution satisfies
$$
\limsup_{p \to \infty} \lambda_p^{ r( \frac{3}{s} + \frac{2}{r}  -1 )}  \int_{t_0-\lambda_p^{-2}}^{t_0} \sum_{q \geq p-2}  \|u_q \|_s^r dt \leq \delta^r   ,
$$
then there exists $\epsilon>0$ such that 
$$
\sup_{[t_0 -\epsilon,t_0]}  \|u  \|_s < \infty.
$$
\end{theorem}

\begin{remark}
As Theorem \ref{main_thm} does not assume \textit{local energy inequality}, it works for Leray-Hopf weak solutions that are not necessarily suitable, which is the main relaxation comparing with the result of Gustafson, Kang, and Tsai \cite{GKT}.
\end{remark}

When $s=r=\frac{10}{3}$ the space $L^{\frac{10}{3}} $ is an interpolation of the energy spaces, for which we obtain a new criterion at the first time of blowup for smooth solutions
\begin{theorem}\label{thm:103}
Suppose $u$ is a Leray-Hopf weak solution on $[0,T)$ with smooth initial data where $T$ is the first time of possible blowup. Then $u$ is regular on $[0,T]$ if u satisfies
$$
\limsup_{p \to \infty} \lambda_{p }^\frac{5}{3}\int_{T - \lambda_p^{-2} }^T \|u(t) \|_{\frac{10}{3}}^{\frac{10}{3}} dt \leq \delta^*
$$
where $\delta^*>0$ is a universal constant.
\end{theorem}

\begin{remark}
Note that any weak solution verifies $\int_{T - \lambda_p^{-2} }^T \|u(t) \|_{\frac{10}{3}}^{\frac{10}{3}} dt \to 0 $ as $p \to \infty$, whereas our condition requires $\int_{T - \lambda_p^{-2} }^T \|u(t) \|_{\frac{10}{3}}^{\frac{10}{3}} dt \leq O(\lambda_{p }^{-\frac{5}{3}}) $
\end{remark}

The paper is organized as follows. In Section 2 we state some preliminaries on properties of weak solutions and the Littlewood-Paley theory. Section 3 is devoted to two main propositions that imply Theorem \ref{main_thm_crite}. We formulate our estimate there using the Besov spaces for optimal results although our argument does not rely on the theory of the Besov spaces. Finally with all ingredients in hand we prove main theorems in Section 4. 

\subsection*{Acknowledgments}
The author would like to express sincere gratitude to his advisor Professor Alexey Cheskidov for proofreading early drafts and giving many suggestions for improvement. The author also acknowledges the partial support form the NSF grant DMS--1517583.

\section{Preliminaries}
\subsection{Notations}
We denote by $A \lesssim B$ an estimate of the form $A \leq CB $ with some
absolute constant $C$, and by $A \sim B$ an estimate of the form $C_1 B \leq  A \leq  C_2 B $ with some absolute constants $C_1$, $C_2$. For any $1\leq p \leq \infty $ we write $\| \cdot \|_p= \|\cdot \|_{L^p} $ for Lebesgue norms. The symbol $(\cdot,\cdot)$ stands for the $L^2$-inner product. For any $p\in \mathbb{N}$ and $t>0$ we let $\lambda_p=2^p$ be the standard dyadic number and $I_p(t)=[t-\lambda_p^{-2},t]$ be the dyadic time interval.

\subsection{Weak solutions}
\begin{definition}
A weak solution to \eqref{3D_NSE} on $[0,T]$(or $(0,\infty)$) with divergence-free initial data $u_0 \in L^2(\mathbb{R})^3$ is a function $u \in C_w(0,T;L^2(\mathbb{R}^3)) \cap L^2(0,T;H^1(\mathbb{R}^3))$ satisfying 
\begin{equation}\label{weak_formulation}
(u(t),\phi(t)) - (u_0,\phi(0))= \int_0^t (u(s),\partial_s \phi(s))+ (\nabla u(s),\nabla\phi(s))+(u(s)\cdot \nabla u(s),\phi(s)) ds,
\end{equation}
$\nabla u(t) =0$ in the sense of distribution for all $t \in [0,T]$ and all divergence-free test functions $\phi \in C^\infty_0([0,T]\times \mathbb{R}^3)$.
\end{definition}

A weak solution that satisfies the energy inequality
\begin{equation}
\|u(t) \|_2^2 +2\int_{t_0}^t \|\nabla u \|_2^2 \leq \| u(t_0)\|_2^2,
\end{equation}
for almost all $t_0 \in (0,T)$ and all $t \in (t_0,T]$ is called a Leray-Hopf weak solution. A major difference between general weak solutions and Leray-Hopf solutions is the weak-strong uniqueness, namely strong solution is unique in the class of Leray-Hopf weak solutions. With this property we only need to consider blowup from the left.

\begin{theorem}[Leray]\label{Leray}
Let $u$ be a Leray-Hopf weak solution of \eqref{3D_NSE}. If $u$ is regular on $[\alpha,\beta)$ and
\begin{equation}
\limsup_{t\to \beta -}\| u(t) \|_p < \infty \quad \text{for some $p>3$}
\end{equation}
then u is regular on $[\alpha,\beta+\epsilon]$ for some small $\epsilon$.
\end{theorem}

\subsection{Littlewood-Paley decomposition}
We introduce a standard Littlewood-Paley decomposition. For a more detailed account of the Littlewood-Paley theory, we refer to \cite{Ca}.

Let $\chi : \mathbb{R}^+ \rightarrow \mathbb{R}$ be a smooth function so that $\chi(\xi) =1$ for $\xi \leq \frac{3}{4}$, and $\chi(\xi) =0$ for $\xi \geq 1$. We further define $\varphi(\xi)=\chi(\lambda_1^{-1 }\xi) -\varphi(\xi) $  and $\varphi_q(\xi) = \varphi(\lambda_q^{-1}\xi)$.For a tempered distribution vector field $u$ let us denote
\begin{align*}
u_q = \mathcal{F}^{-1}(\varphi_q)*u \qquad \text{for } q>-1  \qquad u_{-1} = u_q = \mathcal{F}^{-1}(\chi)*u,
\end{align*}
where $\mathcal{F}$ is the Fourier transform. With this we have $u=\sum_{q \geq -1}u_q$ in the sense of distribution.

Also let us finally note that the Besov space $B^{s}_{p,q}$ is the space consisting of all tempered distributions $u$ satisfying
$$
\| u\|_{B^{s}_{p,q}}:=   \big\| \lambda_r^s \|u_q  \|_p \big\|_{l^q} < \infty.
$$

Finally let us recall the following version of Bernstein's inequality.
\begin{lemma}
Let $u$ be a tempered distribution in $\mathbb{R}^n$, and $r \geq s \geq 1$. Then for any $q\geq -1$ we have that
$$
\| u_q\|_r \lesssim \lambda_q^{n(\frac{1}{s} -\frac{1}{r})} \| u_q\|_r .
$$
\end{lemma}

\section{Regularity away from $\mathcal{S}_T$}
In
\subsection{Definition of singular points}

Recall that $I_p(t) = [t-\lambda_p^{-2},t]$ is the dyadic time interval for any $p \in \mathbb{N}$ .

\begin{definition}
Let $ s > 3$, $r> 2 $ and $u$ be a Leray-Hopf weak solution to \eqref{3D_NSE} on $[0,T]$. For any $p \in \mathbb{N}$ a point $t_0 \subset (0,T]$ is said to be a bad point if 
\begin{equation}\label{bad_points}
\limsup_{p\to \infty}\lambda_p^{  r( \frac{3}{s} + \frac{2}{r}  -1 )} \int_{I_p(t_0)}  \sum_{q\geq p-2}   \| u_q (t) \|_s^r dt \geq \delta^r  ,
\end{equation}
where the constant $\delta $ is defined in Theorem \ref{step2_theorem}.
\end{definition}

The values of $\alpha$, $s$ and $r$ are fixed throughout the note so that there is no confusion in the above definition. Denote $\mathcal{S}_{T}$ the union of bad points on $[0,T]$. A simple covering argument shows the following fractal bound for the set $\mathcal{S}_{T}$.

\begin{lemma}\label{fratals_singualr_set}
Let $u\in L^r_t B^\alpha_{s,\infty}$ for some $\alpha>0$, $ s > 3$ and $r> 2 $. Then 
$$
\mathcal{H}^{\frac{r}{2}(\frac{3}{s} + \frac{2}{r} -\alpha-1 )}(\mathcal{S}_{T})=0.
$$ 
\end{lemma}
\begin{proof}
We observe that thanks to Vitali lemma for each $p\in \mathbb{N}$, $\mathcal{S}_T$ can be covered by finitely many $5I_{p_i}(t_i)$ with $I_{p_i}(t_i)$ being disjoint and $p_i \geq p $ such that
\begin{align}\label{eq:p_i}
\int_{I_{p_i}(t_i)}  \sum_{q\geq p_i-2}   \| u_q (t) \|_s^r dt \geq \delta^r \lambda_{p_i}^{ - r( \frac{3}{s} + \frac{2}{r}  -1 )} .
\end{align}

Since $\alpha>0$, by Jensen's inequality we have that
$$
\int_{I_{p_i}(t_i)}  \sum_{r\geq p_i-2}   \| u_q (t) \|_s^r   dt \lesssim \lambda_{p_i}^{-r\alpha} \int_{I_{p_i}(t_i)}  \sup_{q\geq p_i-2} \lambda_q^{r\alpha}  \| u_q (t) \|_s^r   dt 
$$
which together with \eqref{fratals_singualr_set}  implies that
$$
\lambda_{p_i}^{- r( \frac{3}{s} + \frac{2}{r} -\alpha-1 ) } \leq  \int_{I_{p_i}(t_i)}  \sup_{q \geq p_i-2} \lambda_q^{r\alpha}  \| u_q (t) \|_s^r dt  .
$$
As a result, for the covering we can compute
$$
\sum_{p_i} \lambda_{p_i}^{r(\frac{3}{s} + \frac{2}{r} -\alpha-1 )}\lesssim \sum_{p_i}  \int_{I_{p_i}(t_i)}  \sup_{q\geq p_i-2} \lambda_q^{r\alpha}  \| u_q (t) \|_s^r dt \leq \int_{U_p} \sup_{q\geq p-2 } \lambda_q^{r\alpha}  \| u_q (t) \|_s^r  dt
$$
where $U_p$ is the union of $I_{p_i}$'s  .

By using the fact $u \in L^r_t B^\alpha_{s,\infty}$ and the absolute continuity of Lebesgue integral we obtain $\sum_{p_i} \lambda_{p_i}^{- r( \frac{3}{s} + \frac{2}{r} -\alpha-1 ) }$ goes to $0$ as $p\to \infty$.
\end{proof}
The next proposition is the main tool that we will use in proving Theorem \ref{main_thm_crite}.
\begin{prop}\label{blackbox}
Let $u$ be a weak solution to \eqref{3D_NSE} and $s\geq 2$. Then $\|u_q(t) \|_s$ is absolute continuous and for a.e. $t\in [0,T] $
$$
\frac{d}{dt}\|u_q(t) \|_s +c \lambda_q^2 \|u_q(t) \|_s  \lesssim  \sum_{p\leq q} \lambda_p^\frac{3}{s}\|u_p \|_s \sum_{|p-q|\leq 2} \lambda_p \| u_p\|_s  + \lambda_q^{\frac{3}{s}+1}\sum_{p\geq q-2} \|u_p \|_s^2 .
$$
\end{prop}
\begin{remark}
The proof is just a standard application of the Littlewood-Paley and paraproduct theory. For the sake of completeness we sketch one in the appendix.
\end{remark}

\subsection{Step 1: critical regularity}
The following two results will be used to prove the main theorems. 

\begin{prop}\label{step1_theorem}
Let $ s > 3$ and $r> 2 $. For any $0< \delta<1$, if a Leray-Hopf weak solution verifies the bound
$$
\limsup_{p\to \infty}\lambda_p^{  r( \frac{3}{s} + \frac{2}{r}  -1 )}\int_{I_p(t_0)}  \sum_{q\geq p-2}   \| u_q (t) \|_s^r dt \leq \delta^r  ,
$$
then 
$$
\limsup_{q\to \infty} \lambda_q^{\frac{3}{s}-1}\sup_{ I_q(t_0)}\|u_q \|_s  \lesssim_{s,r}  \delta .
$$
\end{prop}
\begin{proof}
By the definition of $\limsup$ there exist $p_0$ such that for any $p\geq p_0$ 
\begin{equation}\label{step1_p_0}
 \int_{I_p(t_0)}  \sum_{q\geq p-2}   \| u_q (t) \|_s^r dt \leq 2\delta  \lambda_p^{ - r( \frac{3}{s} + \frac{2}{r}  -1 )}.
\end{equation}
Furthermore, there exists $p_1(p_0,\|u_0 \|_2) > p_0$ such that $\lambda_{p_0}^{\frac{3 }{2} }\|u_0 \|_2  \leq \delta\lambda_{p_1}$. We will show 
$$
\sup_{I_p}\|u_p  \|_s         \lesssim \delta   \lambda_p^{1 - \frac{3 }{s}} \quad \text{for all $p \geq p_1$. }
$$

By the Mean Value Theorem for integrals, there exist $t_p \in I_p(t_0)$ such that
\begin{equation}\label{step1_p_1}
\sum_{q\geq p-2}\| u_q  (t_p) \|_s^r  \lesssim \delta  \lambda_p^{  r  -\frac{3r}{s}     }.
\end{equation}

Since $\|u_p(t) \|_s$ is continuous let $t_{p}^*$ be such that $\|u_p(t_p^*) \|_s=\sup_{I_p(t_0)}\|u_p \|_s$.
By Proposition \ref{blackbox} we integrate from $t_p$ to $t_p^*$ for $ \frac{d}{dt}\|u_p(t) \|_s^r$ to find that
\begin{align*}
\sup_{I_p(t_0)} \|u_p \|_s^r - \|u_p(t_p) \|_s^r + c\int_{t_p}^{t_p^*}\lambda_p^2 \|u_p \|_s^r dt  & \lesssim  \int_{t_p}^{t_p^*}\|u_p \|_s^{r-1} \sum_{p'\leq p} \lambda_{p'}^\frac{3}{s}\|u_{p'} \|_s \sum_{|p'-p|\leq 2} \lambda_{p'}\| u_{p'}\|_s   dt \nonumber \\ 
&+ \int_{t_p}^{t_p^*} \lambda_p^{\frac{3}{s}+1}\|u_p \|_s^{r-1} \sum_{p'\geq p-2} \|u_{p'}\|_s^2   dt.
\end{align*}
We use the triangle inequality to obtain that
\begin{align}
\sup_{I_p } \|u_p \|_s^r  - \|u_p(t_p) \|_s^r  \lesssim  &  \int_{I_p }\lambda_p^2 \|u_p \|_s^r  dt \nonumber \\
&+ \lambda_{p }\sup_{I_p } \|u_p \|_s^{r-1}\int_{I_p } \sum_{p'\leq p} \lambda_{p'}^\frac{3}{s}\|u_{p'} \|_s \sum_{|p'-p|\leq 2} \| u_{p'}\|_s   dt \nonumber \\ 
&+   \lambda_p^{\frac{3}{s}+1}\sup_{I_p } \|u_p \|_s\int_{I_p } \sum_{ {p'}\geq p-2} \|u_{p'}\|_s^r  dt \nonumber \\
& :=A+B+C. \label{eq:abc}
\end{align}
For the first term in \eqref{eq:abc}, we have by H\"older's inequality and the assumption that
\begin{equation}
A \leq  \lambda_p^2 \Big[\int_{I_p } \|u_p \|_s^r dt \Big]^\frac{1}{r} \Big[\int_{I_p }1  dt \Big]^\frac{r-1}{r} \lesssim \delta  \lambda_p^{1- \frac{3}{s}}.
\end{equation}
For the second term in \eqref{eq:abc}, by H\"older's inequality with exponents $(2,r,\frac{2r}{ r-2})$ we need to estimate
\begin{align*}
B \leq \lambda_p  \sup_{I_p } \|u_p \|_s^{r-1}\sum_{p'\leq p}\lambda_{p'}^{\frac{3}{s}}\Big[  \int_{I_p}  \|u_{p'} \|_s^2 dt\Big]^\frac{1}{2} \Big[ \int_{I_p}  \sum_{|p'-p|\leq 2}\| u_{p'}\|_s^r dt \Big]^\frac{1}{r} \Big[ \int_{I_p}  1 dt \Big]^\frac{r-2}{ 2r}  .
\end{align*}
By \eqref{step1_p_0} the above can be bounded as
$$
B \lesssim \delta \sup_{I_p } \|u_p \|_s^{r-1}  \lambda_p^{1 - \frac{3}{s}} \sum_{p'\leq p} \lambda_{p'}^{\frac{3}{s}} \Big[  \int_{I_p} \|u_{p'} \|_s^2 dt\Big]^\frac{1}{2}  .
$$
To bound the above, we split the summation to obtain
\begin{align*}
\sum_{p'\leq p} \lambda_{p'}^{\frac{3}{s}} \Big[  \int_{I_p} \|u_{p'} \|_s^2 dt\Big]^\frac{1}{2}  & \leq   \sum_{p'\leq p_0} \lambda_{p'}^{\frac{3}{s}} \Big[  \int_{I_p} \|u_{p'} \|_s^2 dt\Big]^\frac{1}{2}+\sum_{p_0 < p'\leq p} \lambda_{p'}^{\frac{3}{s}} \Big[  \int_{I_p} \|u_{p'} \|_s^2 dt\Big]^\frac{1}{2} .
\end{align*}

By \eqref{step1_p_0} we know that for any $p' \geq p_0$ the bound $
 \int_{I_{p'}}\|u_{p'} \|_s^r dt \lesssim \delta \lambda_{p'}^{ - r( \frac{3}{s} + \frac{2}{r}  -1 )}$ holds. And thus by H\"older's inequality
\begin{align*}
\sum_{p_0 < p'\leq p} \lambda_{p'}^{\frac{3}{s}} \big[  \int_{I_p} \|u_{p'} \|_s^2 dt\big]^\frac{1}{2}  & \lesssim \sum_{p_0 < p'\leq p} \lambda_{p'}^{\frac{3}{s}} \big[  \int_{I_p} \|u_{p'} \|_s^r dt\big]^\frac{1}{r} \lambda_p^{-1+ \frac{2}{r}}\\
&  \lesssim \delta\sum_{p_0 < p'\leq p} \lambda_{p'}^{1 -\frac{2}{r}  }   \lambda_p^{-1+ \frac{2}{r}}\lesssim \delta
\end{align*}
where we have used $ r >2$.
Using the Bernstein inequality, the energy bound $\|u(t) \|_2 \leq \|u_0 \|_2$ and the definition of $p_1$ we obtain 
$$
\sum_{  p'\leq  p_0 } \lambda_{p'}^{\frac{3}{s}} \Big[  \int_{I_p} \|u_{p'} \|_s^2 dt\Big]^\frac{1}{2} \lesssim \sum_{  p'\leq p_0} \lambda_{p'}^{\frac{3 }{2} } \|u_0 \|_2 \lambda_{p}^{-1} \leq \lambda_{p_0}^{\frac{3 }{2} } \|u_0 \|_2 \lambda_{p_1}^{-1} \leq \delta .
$$

Putting together the split summation we have
$$
\sum_{p'\leq p} \lambda_{p'}^{\frac{3}{s}} \Big[  \int_{I_p} \|u_{p'} \|_s^2 dt\Big]^\frac{1}{2}    \lesssim \delta.
$$

So the term $B$ verifies the bound:
$$
B \lesssim \delta   \sup_{I_p } \|u_p \|_s^{r-1}  \lambda_p^{1 - \frac{3}{s}}    .     
$$
Next, the estimate for the term $C$ directly follows from \eqref{step1_p_0}:
$$ 
C\lesssim   \delta^r  \lambda_p^{r(1 - \frac{3}{s})} \sup_{I_p } \|u_p \|_s. 
$$
 
Putting together \eqref{step1_p_1} and the estimates for $A$, $B$, $C$ we have
$$
\sup_{I_p}\|u_p  \|_s^r         \lesssim \delta^r   \lambda_p^{r1 - \frac{3r }{s}} + \delta   \sup_{I_p } \|u_p \|_s^{r-1}  \lambda_p^{1 - \frac{3}{s}} +\delta^r  \lambda_p^{r(1 - \frac{3}{s})} \sup_{I_p } \|u_p \|_s.
$$

Using for example Young's inequality finishes the proof.
\end{proof}
\subsection{Step 2: Local-in-time regularity}
The regularity in Proposition \ref{step1_theorem} is not enough to obtain the smoothness of $u$. We will close this gap by a continuity argument.

\begin{theorem} \label{step2_theorem}
For any $ s > 3$ and $r> 2 $, there exist a constant $\delta  >0$ with the following property: if a Leray-Hopf weak solution satisfies
$$
\limsup_{p \to \infty} \lambda_p^{  r( \frac{3}{s} + \frac{2}{r}  -1 )} \int_{I_p(t_0)} \sum_{q \geq p-2} \|u_q \|_s^r dt \leq \delta^r  
$$
then there exist a integer $p>0$ and an interval $[\tau_p,t_0]\subset I_{p}(t_0)$ such that 
$$
\sup_{[\tau_p,t_0]}\sum_{q \geq p-2} \|u_q \|_s^r \lesssim \lambda_p^{  r(1- \frac{3}{s}  )}.
$$
\end{theorem}
\begin{proof}
The exact value of $\delta $ will be chosen in the end. 

First of all by Proposition \ref{step1_theorem} there exists $p_0$ such that for any $p \geq p_0$ the following 2 conditions hold
\begin{align}
  \int_{I_p } \sum_{q \geq p-2} \|u_q \|_s^r dt &\leq 2\delta \lambda_p^{  -r( \frac{3}{s} + \frac{2}{r}  -1 )},\label{step2_2conditions}\\
\sup_{ I_p }\|u_p \|_s  &\lesssim \delta \lambda_p^{1-\frac{3}{s} }\label{step2_2conditions_2}.
\end{align}

To handle the low modes errors in the later estimates, we introduce the lower bound $p_1=p_1(p_0,\delta,\|u_0 \|_2) \geq p_0$ so that 
\begin{equation}\label{step2_p1}
\lambda_{p_0}^{\frac{3}{s} +(r-1)(\frac{1}{2}  -\frac{1}{s}) } \|u_0 \|_2^{s-1} \leq \delta^{r-1}  \lambda_{p_1}^{ \frac{3}{s} } \lambda_{p_1}^{(r-1)(1-\frac{3}{s})}.
\end{equation} 

We fix some $p \geq p_1$ and will show the bound $
\sum_{q \geq p-2} \|u_q  \|_s^r \lesssim \delta^r \lambda_p^{r (1-\frac{3}{s} )}
$ holds on some interval up to $t_0$. 

By the first condition \eqref{step2_2conditions} there exists $\tau_p \in [t-\lambda_p^{-2}, t_0)$ such that
\begin{equation}
 \sum_{q \geq p-2} \|u_q(\tau_p) \|_s^r   \leq 2\delta^r  \lambda_p^{r (1-\frac{3}{s} )},
\end{equation}
i.e. the desired bound is satisfied.

By local existence and uniqueness theory for Leray-Hopf weak solutions in $L^s$ for $s>3$ (cf. Theorem \ref{Leray}), there exists an nonempty interval $[\tau_p,t_p]$ on which 
\begin{equation}\label{step2_interval}
\sum_{q \geq p-2} \|u_q \|_s^r   \leq 4\delta^r \lambda_p^{r (1-\frac{3}{s} )}.
\end{equation}
Next we will use a continuity argument to show that if the above inequality holds on the interval $[\tau_p,t_p]$, then $\sum_{q \geq p-2} \|u_q(t_p) \|_s^r   < 3\delta^r \lambda_p^{r (1-\frac{3}{s} )}$.

Consider the equation for $u_q$ on $[\tau_p,t_p]$ for every $q \geq p-2$ in the following form:
\begin{equation}\label{step2_eq_at_breakdown}
\frac{d}{dt}\|u_q \|_s^{r-1}  +c \lambda_q^2 \|u_q \|_s^{r-1}  \lesssim    \lambda_q\sum_{p'\leq q} \lambda_{p'}^{\frac{3}{s}}\|u_{p'} \|_s^{r-1}  \sum_{|p'-q|\leq 2} \| u_{p'}\|_s   + \lambda_q^{\frac{3}{s}+1}\sum_{p' \geq q-2} \|u_{p'} \|_s^{r } .
\end{equation}

We will bound the terms on the right hand side of \eqref{step2_eq_at_breakdown} on the interval $[\tau_p,t_p]$.

For the first term in \eqref{step2_eq_at_breakdown}, we consider the split
$$
\sum_{p'\leq q} \lambda_{p'}^{\frac{3}{s}} \|u_{p'} \|_s^{r-1} \leq \sum_{p-2 \leq p' \leq q} \lambda_{p'}^{\frac{3}{s}} \|u_{p'} \|_s^{r-1} +\sum_{p_0 \leq p'\leq p-2} \lambda_{p'}^{\frac{3}{s}} \|u_{p'} \|_s^{r-1} +\sum_{p'\leq p_0} \lambda_{p'}^{\frac{3}{s}} \|u_{p'} \|_s^{r-1}.
$$

The idea is to bound modes below $p_0$ by energy, modes between $p_0$ and $p-2$ by critical regularity and modes above $p-2$ by our hypothesis.

For the last part the Bernstein inequality, the energy inequality and the definition of $p_0$ and $p_1$ imply that 
$$
\sum_{p'\leq p_0} \lambda_{p'}^{\frac{3}{s}} \|u_{p'} \|_s^{r-1} \lesssim \lambda_{p_0}^{\frac{3}{s} +(r-1)(\frac{1}{2}  -\frac{1}{s}) } \|u_0 \|_2^{s-1} \leq \delta^{r-1}  \lambda_p^{ \frac{3}{s} } \lambda_p^{(r-1)(1-\frac{3}{s})}.
$$
By the Jensen inequality we have 
$$
\sum_{p-2 \leq p' \leq q} \lambda_{p'}^{\frac{3}{s}} \|u_{p'} \|_s^{r-1} \leq \lambda_q^\frac{3}{s} \big[ \sum_{p-2 \leq p' \leq q} \|u_{p'} \|_s^r \big]^\frac{r-1}{r}  \lesssim \delta^{r-1}  \lambda_q^{ \frac{3}{s} } \lambda_p^{(r-1)(1-\frac{3}{s})}
$$ for the first part. 
From \eqref{step2_2conditions_2} it follows that  
$$
\sum_{p_0 \leq p'\leq p-2} \lambda_{p'}^{\frac{3}{s}} \|u_{p'} \|_s^{r-1}  \lesssim  \delta^{r-1}  \lambda_p^{ \frac{3}{s} } \lambda_p^{(r-1)(1-\frac{3}{s})} .
$$

By \eqref{step2_interval} we have $\sum_{|p'-q|\leq 2} \| u_{p'}\|_s \lesssim \delta \lambda_{p}^{1-\frac{3}{s}}$.

Combining these two estimates, the first part of the nonlinear term verifies
\begin{equation}\label{low_part_nonlinear}
\lambda_q\sum_{p'\leq q} \lambda_{p'}^{\frac{3}{s}}\|u_{p'} \|_s^{r-1}  \sum_{|p'-q|\leq 2} \| u_{p'}\|_s^{}  \lesssim \delta^{r } \lambda_q^{ \frac{3}{s}+1 } \lambda_p^{r(1-\frac{3}{s})} .
\end{equation}

For the last term on the right of \eqref{step2_eq_at_breakdown} we once again use \eqref{step2_interval} to obtain
\begin{equation}\label{high_part_nonlinear}
 \lambda_q^{\frac{3}{s}+1}\sum_{p'\geq q-2} \|u_{p'} \|_s^{r }  \lesssim \delta^{r } \lambda_q^{ \frac{3}{s}+1 } \lambda_p^{r( 1-\frac{3}{s})} .
\end{equation}

Putting \eqref{low_part_nonlinear} and \eqref{high_part_nonlinear} together we obtain on $[\tau_p,t_p]$ the differential inequality:
$$
\frac{d}{dt}\|u_q \|_s^{r-1}+c\lambda_q^2 \|u_q \|_s^{r-1} \lesssim \delta^r \lambda_q^{ \frac{3}{s}+1 } \lambda_p^{r( 1-\frac{3}{s})} .
$$
By the Gronwall inequality we have that
\begin{equation}\label{step2_final_gronwall}
\|u_q(t_p) \|_s^{r-1} \leq \|u_q(\tau_p) \|_s^{r-1} e^{-c\lambda_q^2(t-\tau_p)} + C\big[1-e^{-c\lambda_q^2(t-\tau_p)} \big]  \delta^r \lambda_q^{ \frac{3}{s} -1 } \lambda_p^{ r (1-\frac{3}{s})} ,
\end{equation}
for every $q \geq p-2$, where $C>0$ is a   constant depending on $s$ and $r$.

Let $M$ be a sufficiently large constant depending on $s>3$ and $r>2$ so that
$$
\sum_{q\geq p-2 } \big(  \lambda_q^{ \frac{3}{s} -1 } \lambda_p^{r(1-\frac{3}{s})} \big)^{\frac{r}{r-1}} \leq   M  \lambda_p^{ r(1 - \frac{3}{s})}.
$$
We choose $0<\delta \leq \frac{1}{64M C}$ and define the index set $\mathcal{I}_p\subset \mathbb{Z}$ in the following manner:
\begin{equation}\label{step2_indexset}
\mathcal{I}_p : =\{q: q \geq p-2  \text{  and  } \|u_q(\tau_p) \|_s^{r-1} \geq \frac{1}{8M}\delta^{r-1}\lambda_q^{ \frac{3}{s} -1 } \lambda_p^{ r (1-\frac{3}{s})}  \}.
\end{equation}

From this we have the following decomposition:
$$
\sum_{q \geq p-2} \|u_q(t_p) \|_s^r = \sum_{q\in \mathcal{I}_p } \|u_q (t_p)\|_s^r +  \sum_{q\in \mathcal{I}_p^C } \|u_q (t_p) \|_s^r.
$$

On the one hand, for $q \in \mathcal{I}_p $ by \eqref{step2_final_gronwall} and \eqref{step2_indexset} we obtain
\begin{align*}
\|u_q(t_p) \|_s^{r-1} & \leq \|u_q(\tau_p) \|_s^{r-1} e^{-c\lambda_q^2(t_p-\tau_p)} + \frac{1}{8M}\big[1-e^{-c\lambda_q^2(t_p-\tau_p)} \big] \|u_q(\tau_p) \|_s^{r-1} \\
&\leq \frac{9}{8 } \|u_q(\tau_p) \|_s^{r-1}.
\end{align*}

Taking a summation in $\mathcal{I}_p$ yields
$$
\sum_{q\in \mathcal{I}_p } \|u_q (t_p)\|_s^r \leq \frac{9 }{8} \sum_{q\in \mathcal{I}_p } \|u_q (\tau_p)\|_s^r \leq \frac{9 }{4 }\delta^r \lambda_p^{r(1- \frac{3}{s})},
$$
where we have used the fact that $\sum_{q } \|u_q (\tau_p)\|_s^r \leq 2 \delta^r \lambda_p^{r(1- \frac{3}{s})}$.
   
On the other hand, for $q \not\in \mathcal{I}_p $ once again by \eqref{step2_final_gronwall} and \eqref{step2_indexset} we obtain
\begin{align*}
\|u_q(t_p) \|_s^{r-1} & < \frac{1}{8M}\delta^{r-1}\lambda_q^{ \frac{3}{s} -1 } \lambda_p^{ r (1-\frac{3}{s})} e^{-c\lambda_q^2(t_p-\tau_p)} + \frac{1}{64 M}\big[1-e^{-c\lambda_q^2(t_p-\tau_p)} \big]\delta^{r-1} \lambda_q^{ \frac{3}{s} -1 } \lambda_p^{ r (1-\frac{3}{s})} \\
& \leq \frac{9}{64M}\delta^{r-1}\lambda_q^{ \frac{3}{s} -1 } \lambda_p^{ r (1-\frac{3}{s})} .
\end{align*}

Taking a summation in $\mathcal{I}_p^C$ and using the definition of $M$ yield
$$
\sum_{q\in \mathcal{I}_p^C } \|u_q (t_p)\|_s^r \leq \frac{9}{64M}  \delta^{r }\sum_{q\in \mathcal{I}_p^C } \big(  \lambda_q^{ \frac{3}{s} -1} \lambda_p^{r(1-\frac{3}{s})} \big)^{\frac{r}{r-1}}\leq \frac{9}{64 } \delta^r    \lambda_p^{ r(1 - \frac{3}{s})}.
$$

Combining the decomposition it follows that 
$$
\sum_{q \geq p-2} \|u_q(t_p) \|_s^r  < 3\delta^r    \lambda_p^{ r(1 - \frac{3}{s})}.
$$
And hence an iteration of applying local regularity result and the above continuity argument yields the desire bound:
$
\sup_{[\tau_p,t_0]}\sum_{q \geq p-2} \|u_q \|_s^r  \leq 4\delta^r    \lambda_p^{ r(1 - \frac{3}{s})}.
$
\end{proof}

\section{Proof of main results}
Thanks to the above two theorems, we can prove our results stated in the introduction.

\begin{proof}[Proof of Theorem \ref{main_thm}]
By Proposition \ref{step1_theorem} and Theorem \ref{step2_theorem}, we know that if $t_0\not \in \mathcal{S}_T$  there exists a small $\epsilon>0$ such that $u \in L^\infty(t_0-\epsilon,t_0;L^s)$. 

Since $s>3$, the space $  L^\infty(t_0-\epsilon,t_0;L^s)$ is subcritical to the Navier-Stokes
scaling. We can use for instance classical Serrin’s regularity result to bootstrap
arbitrary regularity and obtain $u \in C^\infty((t_0 - \epsilon,t_0 )\times\mathbb{R}^3)$
 
Therefore by local regularity result for Leray-Hopf weak solutions we can assert
$u \in C^\infty((t_0 - \epsilon',t_0 +\epsilon')\times\mathbb{R}^3) $ for some small $\epsilon'>0$.
 
\end{proof}

Theorem \ref{main_thm_crite} follows from Proposition \ref{step1_theorem} and Theorem \ref{step2_theorem} while Theorem \ref{thm:103} is a direct consequence of the embedding $  L^{\frac{10}{3}}   \subset B^0_{\frac{10}{3} , \frac{10}{3}} $ and Theorem \ref{main_thm_crite}.
\appendix
\section{Proof of Proposition \ref{blackbox}}
We only prove the estimates for strong solutions. To prove the validity for general weak solutions one can use \eqref{weak_formulation} in the class of divergence-free Schwartz functions.

Let $\mathbb{P}$ be the Leray projection. Multiplying \eqref{3D_NSE} by $s \mathbb{P}\Delta_q(u_{q}|u_q|^{s-2})$ and integrating in space yields
\begin{equation}
\frac{d}{dt}\|u_q \|^s_s +  s\int  \Delta u_q u_{q}|u_q|^{s-2} dx = -s\int \mathbb{P} \Delta_q(u\cdot \nabla u) u_q |u_q|^{s-2} dx .
\end{equation}
Note that we have used the fact that $\mathbb{P}u_q =u_q$.

It is known that  $\int  \Delta u_q u_{q}|u_q|^{s-2} dx \sim \lambda_q^2 \|u_q \|_s^s$. We also use the following version of paraproduct decomposition:
\begin{equation*}
\Delta_q(u\cdot v) = \sum_{p:|p-q|\leq 2}\Delta_q(u_{\leq p-2 }\cdot v_p) + \sum_{p:|p-q|\leq 2}\Delta_q(u_{p }\cdot v_{\leq p-2}) +\sum_{p:p\geq q-2}\Delta_q(\tilde{u}_{p }\cdot v_p) .
\end{equation*}
From the above two facts it follows that
\begin{equation*}
\frac{d}{dt}\|u_q \|^s_s +   \lambda_q^2 \|u_q \|_s^s \leq I_1+I_2+I_3
\end{equation*}
where 
\begin{equation*}
I_1 \sim  \Big| \int\sum_{p:|p-q|\leq 2}\mathbb{P}\Delta_q(u_{\leq p-2 }\cdot \nabla u_p) u_q |u_q|^{s-2}  dx \Big|,
\end{equation*}
\begin{equation*}
I_2 \sim \Big| \int \sum_{p:|p-q|\leq 2}\mathbb{P}\Delta_q(u_{p }\cdot \nabla u_{\leq p-2}) u_q |u_q|^{s-2}  dx \Big|,
\end{equation*}
and
\begin{equation*}
I_3 \sim \Big| \int \sum_{p:p\geq q-2}\mathbb{P}\Delta_q(\tilde{u}_{p }\cdot \nabla u_p)  u_q |u_q|^{s-2}  dx \Big|.
\end{equation*}

By the H\"older inequality and the boundedness of the operator $\mathbb{P}\Delta_q $ we find:
\begin{equation*}
I_1 \lesssim  \sum_{p:|p-q|\leq 2} \|u_{\leq p-2 }\cdot \nabla u_p\|_{s}   \| u_q\|_s^{s-1}.
\end{equation*}
It can be further bounded by
$$
\lesssim \sum_{p' \leq q}    \|u_{p'} \|_\infty \sum_{p:|p-q|\leq 2}  \| \nabla u_p\|_{s}   \| u_q\|_s^{s-1}.
$$
Thus the Bernstein inequality gives:
\begin{equation}
I_1 \lesssim \sum_{p' \leq q}   \lambda_{p'}^{\frac{3}{s}} \|u_{p'} \|_s \sum_{p:|p-q|\leq 2}  \lambda_p \|   u_p\|_{s}   \| u_q\|_s^{s-1}.
\end{equation}

For the second term $I_2$ the H\"older inequality yields 
\begin{equation*}
I_2 \lesssim \sum_{p:|p-q|\leq 2} \|u_p\|_{s} \| \nabla u_{\leq p-2 }\|_\infty     \| u_q\|_s^{s-1}.
\end{equation*}

The Bernstein inequality now gives:
\begin{equation}
I_2 \lesssim \sum_{p:|p-q|\leq 2} \|u_p\|_{s}  \sum_{p' \leq q} \lambda_{p'}^{\frac{3}{s}+1}\|   u_{p'}\|_s    \| u_q\|_s^{s-1}.
\end{equation}

Finally for the last term $I_3$ we integrate by parts to obtain:
\begin{equation*}
I_3 \lesssim   \Big| \int \sum_{p:p\geq q-2} \mathbb{P}\Delta_q(\tilde{u}_{p }\cdot  u_p)  \nabla( u_q |u_q|^{s-2}) dx \Big|   .
\end{equation*}
Direct computations and the H\"older inequality yield:
\begin{equation*}
I_3 \lesssim      \sum_{p:p\geq q-2} \| \tilde{u}_{p } \otimes  u_p\|_{\frac{s}{2}}  \| \nabla  u_q\|_\infty  \| u_q  \|_s^{s-2}    .
\end{equation*}
So we can obtain the desire bound:
\begin{equation}
I_3 \lesssim     \lambda_q^{1+\frac{3}{s}}  \sum_{p:p\geq q-3} \|  u_{p }\|_s^2      \| u_q  \|_s^{s-1}    .
\end{equation}

Putting the bounds for $I_1$, $I_2$ and $I_3$ together and dividing a common factor $\|u_q \|_s^{s-1}$ we have
\begin{equation*}
\frac{d}{dt}\|u_q(t) \|_s   +\lambda_q^2 \|u_q(t) \|_s   \lesssim  \sum_{p\leq q} \lambda_p^\frac{3}{s}\|u_p \|_s \sum_{|p-q|\leq 2} \lambda_p \| u_p\|_s   + \lambda_q^{\frac{3}{s}+1}\sum_{p\geq q-2} \|u_p \|_s^2 .
\end{equation*}

\end{document}